\documentclass[11pt,reqno]{amsart}
\usepackage{amssymb,latexsym}
\usepackage{graphicx}
\usepackage{pdflscape}
\usepackage{tikz}
\usepackage{circuitikz}
\usetikzlibrary{arrows}
\usepackage{url}		

\calclayout

\makeatletter
\newcommand{\monthyear}[1]{%
  \def\@monthyear{\uppercase{#1}}}
\newcommand{\volnumber}[1]{%
  \def\@volnumber{\uppercase{#1}}}
\AtBeginDocument{%
\def\ps@plain{\ps@empty
  \def\@oddfoot{\@monthyear \hfil \thepage}%
  \def\@evenfoot{\thepage \hfil \@volnumber}}
\def\ps@firstpage{\ps@plain}
\def\ps@headings{\ps@empty
  \def\@evenhead{%
    \setTrue{runhead}%
    \def\thanks{\protect\thanks@warning}%
    \uppercase{The Fibonacci Quarterly}\hfil}%
  \def\@oddhead{%
    \setTrue{runhead}%
    \def\thanks{\protect\thanks@warning}%
    \hfill\uppercase{RECURSIONS OF THE CIRCUIT ARRAY}}%
  \let\@mkboth\markboth
  \def\@evenfoot{%
    \thepage \hfil \@volnumber}%
  \def\@oddfoot{%
    \@monthyear \hfil \thepage}%
  }%
\footskip=25pt
\pagestyle{headings}%
}
\makeatother

\theoremstyle{plain}
\numberwithin{equation}{section}
\newtheorem{thm}{Theorem}[section]

\newtheorem{lemma}[thm]{Lemma}
\newtheorem{example}[thm]{Example}
\newtheorem{definition}[thm]{Definition}
\newtheorem{proposition}[thm]{Proposition}

\newtheorem{remark}[thm]{Remark}
\newtheorem{conjecture}[thm]{Conjecture}
\newtheorem{corollary}[thm]{Corollary}

\begin{document}
\monthyear{Month Year}
\volnumber{Volume, Number}
\setcounter{page}{1}

\title[Circuit Array Polynomials]{Recursions and characteristic polynomials of the rows of the circuit array}
\author{Emily J. Evans}
\address{Brigham Young University}
\email{EJEvans@math.byu.edu}
\author{Russell Jay Hendel}
\address{Towson University}
\email{RHendel@Towson.Edu}

\thanks{Acknowledgement is given to an anonymous referee for a very thorough, detailed, and comprehensive review of the manuscript.}

\begin{abstract}
This paper extends our previous result on the  circuit array, a two-dimensional array associated with the resistances in circuits whose underlying graph when embedded in the Cartesian plane has the form of a triangular $n$-grid. This paper extends the results of the prior paper by considering the circuit array in terms of polynomials instead of numbers as a means to facilitate finding patterns.  The main discovery of this paper states that the characteristic polynomials corresponding to the recursions of single or multi-variable polynomial formulations of the circuit array exclusively have powers of 9 as roots. Several initial cases and one major sub-case are proven.    
\end{abstract}

\maketitle


\section{Introduction}\label{sec:introduction}

In prior work~\cite{Sarajevo}, the authors introduced the circuit array, a two-dimensional array associated with the resistance values in electrical circuits whose underlying graphs when embedded in the Cartesian plane have the form of a triangular $n$-grid. The circuit array is constructed by applying to an initial electric circuit with specified resistances a sequence of equivalent network transformations.  The prior work showed certain patterns in the circuit array that formed linear homogeneous recursions with constant coefficients and suggested several approaches for studying further patterns in the sequences of the rows of the circuit array. The main result of this paper states that the characteristic polynomial (annihilator) of certain rows $i, i \ge 0,$ of the circuit array, when it is re-formulated in terms of single or multi-variable polynomials, exclusively has roots that are powers of 9.  



We begin by briefly reviewing the history of approaches to (exact) computations of effective resistance in a graph.  First, we recall that effective resistance, also termed in the literature resistance distance, is a graph metric whose definition was motivated by the consideration of a graph as an electrical circuit.  A formal definition of this metric can be found in~\cite{Sarajevo} but the effective resistance $r(i,j)$ between nodes $i$ and $j$ of a graph is the resistance measured between those nodes where each link in the graph has a given resistance (typically one ohm).

 The initial approach to computing resistances was to use the Kirchhoff laws which at least theoretically allowed solving any electrical circuit. However, for a circuit whose underlying graph is embeddable in the Cartesian plane, all computations and simplifications of this circuit can be made with four well-known circuit transformation functions: parallel, series, $\Delta$--Y and Y--$\Delta.$ These four circuit transformation functions reduce the complexity of the underlying graph while leaving the effective resistance between two vertices of interest the same.

For circuits derived from small graphs or graphs with special structures, the combinatorial Laplacian is frequently used to compute resistances between two points in a circuit~\cite{Bapat0,Barrett9}. For triangular $n$-grids, \cite{Barrett9, Hendel} exploit the idea of reducing the number of rows in the underlying, initial, triangular $n$-grid one row at a time. This can be accomplished by the four electrical circuit functions. \cite{Barrett9} used this idea to accomplish proofs while \cite{Hendel} was the first to show its computational advantages in discovering patterns.

To simplify the computations of resistances arising from reductions, \cite{Sarajevo, EvansHendel} showed that only four local functions, each of which takes up to nine arguments, need to be considered. These four functions calculate resistances for a (i) left-boundary, (ii) left-interior, (iii) right, and (iv) base edge of a given triangle in a triangular $n$-grid reduced once.

 An outline of this paper is the following.
  Section~\ref{sec:background} presents necessary terminology, notation, and conventions; introduces the four transformation functions; and defines the various circuit arrays. Familiarity with the prior paper \cite{Sarajevo} is not assumed; a modest example illustrates basic computations and patterns. Section \ref{sec:annihilators} reviews the method of annihilators. The main results of the paper for single-variable polynomials and multi-variable polynomials is presented in Section \ref{sec:proofCX}. Generalization of these results are conjectured with certain partial results proven.  The paper closes with the observation that the patterns studied in this paper should prove fertile ground for future  researchers.
 
\section{Background, Notation, Conventions, Functions, and Arrays}\label{sec:background}
  
In this section, we  review important terminology and concepts.    The material in this section is taken almost verbatim from~\cite{Sarajevo, EvansHendel, Hendel}; therefore, attribution to the original and subsequent sources are omitted below except for a few important concepts.\\

\begin{definition}\label{def:ngrid}
A $n$-triangular-grid (usually abbreviated as an $n$-grid)
is any graph that is (graph-) isomorphic to the graph,
whose vertices are all integer pairs $(x,y)=(2r+s,s)$ 
in the Cartesian plane, with $r$ and $s$ integer
parameters satisfying
$0 \le r \le n, 0 \le s \le n-r;$
  and whose edges
consist of any two vertices $(x,y)$ and $(x',y')$ with
(i) $x'-x=1, y'-y=1,$ 
 (ii) $x'-x=2, y'-y=0,$ or
(iii) $x'-x=1, y'-y=-1.$   
\end{definition}

\begin{figure}[ht!]
\begin{center}
\begin{tabular}{|c|}
\hline
\begin{tikzpicture} [xscale=1,yscale =1]

\node [below] at (3,3.2) {    }; 
 
\draw (0,0)--(1,1)--(2,2)--(3,3);
\draw (2,0)--(3,1)--(4,2);
\draw (4,0)--(5,1);

\draw (3,3)--(4,2)--(5,1)--(6,0);
\draw (2,2)--(3,1)--(4,0);
\draw (1,1)--(2,0);

\draw (0,0)--(2,0)--(4,0)--(6,0);
\draw (1,1)--(3,1)--(5,1);

\draw (2,2) -- (4,2);

\begin{scriptsize}
\node [below] at (3,2.8) {$(3,3)$};

\node [below] at (2,1.8) {$(2,2)$};
\node [below] at (4,1.8) {$(4,2)$};

\node [below] at (1,.8) {$(1,1)$};
\node [below] at (3,.8) {$(3,1)$};
\node [below] at (5,.8) {$(5,1)$};

\node [below] at (6,0) {$(6,0)$};
\node [below] at (4,0) {$(4,0)$};
\node [below] at (2,0) {$(2,0)$};
\node [below] at (0,0) {$(0,0)$};

\end{scriptsize}
\end{tikzpicture}

\begin{tikzpicture} [xscale=1,yscale =1]

\node [below] at (3,3.2) {    }; 
 
\draw (0,0)--(1,1)--(2,2)--(3,3);
\draw (2,0)--(3,1)--(4,2);
\draw (4,0)--(5,1);

\draw (3,3)--(4,2)--(5,1)--(6,0);
\draw (2,2)--(3,1)--(4,0);
\draw (1,1)--(2,0);

\draw (0,0)--(2,0)--(4,0)--(6,0);
\draw (1,1)--(3,1)--(5,1);

\draw (2,2) -- (4,2);

\begin{scriptsize}
\node  [above] at (1,.2) {$T_{3,1}$};  
\node  [above] at (3,.2) {$T_{3,2}$};
\node  [above] at (5,.2) {$T_{3,3}$};
\node  [above] at (2,1.2) {$T_{2,1}$};
\node  [above] at (4,1.2) {$T_{2,2}$};
\node  [above] at (3,2.3) {$T_{1,1}$};

\node [below] at (6,0) {$ $};
\node [below] at (4,0) {$ $};
\node [below] at (2,0) {$ $};
\node [below] at (0,0) {$ $};

\end{scriptsize}

\end{tikzpicture}
\\  \hline
\end{tabular}
\caption{A 3-grid embedded in the Cartesian Plane constructed using Definition~\ref{def:ngrid} (Left Panel) with row and diagonal coordinates (Right Panel).}\label{fig:3grid}
\end{center}
\end{figure}
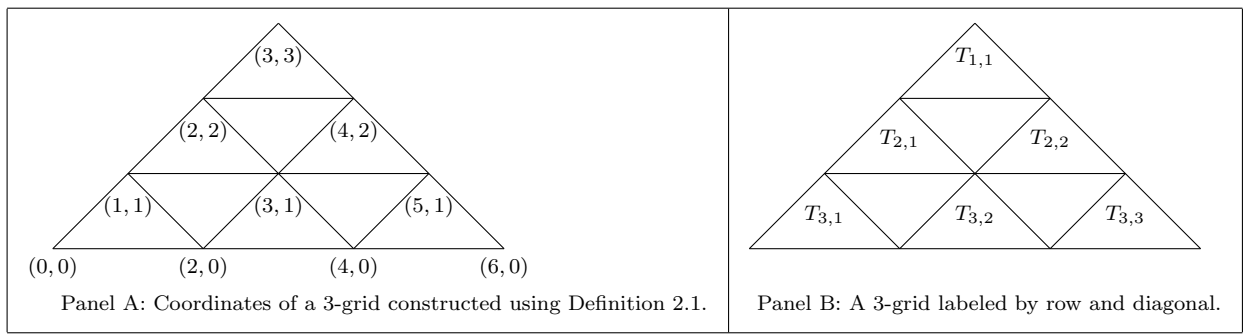

Figure \ref{fig:3grid} illustrates this definition along with various notational conventions which are thoroughly explained below. \\

As mentioned in the introduction, the focus of this paper is on recursions, and not the underlying circuit theory.  The recursions studied in this paper arise by applying three circuit transformations to the $n$--grid:  the well-known series rule (from physics), the $\Delta$--Y, and the Y-$\Delta$ transformations.  The $\Delta$--Y transformation is a mathematical technique to convert resistors in a triangle ($\Delta$) formation to an equivalent system of three resistors in a ``Y'' format as illustrated in Figure~\ref{fig:DeltaY}. More formally, the $\Delta$--Y and Y--$\Delta$ functions defined from circuit transformations are given by:

 \begin{equation}\label{equ:DeltaY}
 \Delta(x,y,z) = \frac{xy}{x+y+z}; \qquad
 Y(a,b,c) = \frac{ab+bc+ca}{a}.
 \end{equation}

Conventions on the order of the arguments are illustrated in Figure \ref{fig:DeltaY}.

 
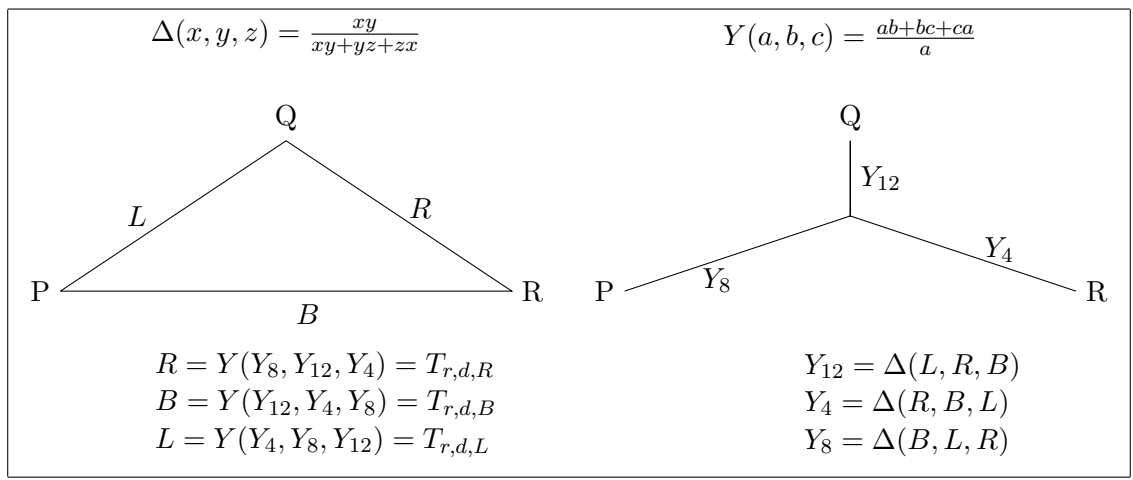
\begin{figure}[ht!]
\begin{center}
\begin{tabular}{|c|}
\hline

\begin{tikzpicture}[yscale=.5,xscale=1.5]

 
\draw (0,0)--(2,4)--(4,0)--(0,0) ;
\draw (5,0)--(7,2)--(7,4)--(7,2)--(9,0);

\node [below] at (7,4.2) { };

\node [above] at (7,4) {Q};
\node [left] at (5,0) {P};
\node [right] at (9,0) {S};

\node [above] at (2,4) {Q};
\node [left] at (0,0) {P};
\node [right] at (4,0) {S};

\node [right] at (7,3) {$Y_{12}$};
\node [right] at (8.1,1.1) {$Y_{4}$};
\node [right] at (5.6,.3) {$Y_{8}$}; 

\node [right] at (3,2.2) {$R$};
\node [right] at (2,-.6) {$B$};
\node [right] at (.5,2) {$L$}; 

\node [right] at (6.5,-2) {$Y_{12}=\Delta(L,R,B)$};
\node [right] at (6.5,-3) {$Y_{4}=\Delta(R,B,L)$};
\node [right] at (6.5,-4) {$Y_{8}=\Delta(B,L,R)$}; 

\node [right] at (.75, -2) {$R=Y(Y_8,Y_{12},Y_4)$};
\node [right] at (.75,-3) {$B=Y(Y_{12},Y_{4},Y_8)$};
\node [right] at (.75,-4) {$L=Y(Y_4,Y_8,Y_{12})$};

\end{tikzpicture} 
\\
\hline

\end{tabular}

\caption{Definition and assumed orientation of the 
$\Delta$ (that is $\Delta$--Y) and $Y$ (that is the Y--$\Delta$) resistance functions. $P, Q, S$ are vertex labels; $L, R, B$ are edge labels standing for left, right, and bottom edge.    We use e.g. $L$ to refer both to the left edge and the label of the left edge; the meaning should be clear from the context. 
}\label{fig:DeltaY}

\end{center}
\end{figure}

As mentioned in the introduction, \textit{row reduction} is an algorithm, formulated using series, $\Delta$--Y, and Y--$\Delta$ transformations, that takes an $n$-grid and creates an $n-1$ grid.  Given an $n$-grid with given resistance values, the   
\emph{row-reduction} of this $n$-grid (to a $n-1$ grid) 
refers to the sequential performance of the following steps (with illustrations of the steps provided by Figure \ref{fig:5panels}).
\begin{itemize}
    \item Step 1: Start with an \textit{$n$-grid} (Figure \ref{fig:5panels} illustrates with $n=3$). 
    \item Step 2: Apply a $\Delta$--Y transformation to each upright triangle (a 3-loop) resulting in a grid of $n$ rows of out-stars.
    \item Step 3: Discard the corner tails, that is, edges with a vertex of degree one. This does not affect the resistances in the reduced two grid shown in Step 5. (However, these corner tails are useful for computing certain effective resistances as shown in 
    \cite{Barrett0,Evans2022}).
    \item Step 4: Perform series transformations on all consecutive pairs of boundary edges (i.e., the dashed edges in Step 3).
    \item Step 5: Apply Y--$\Delta$ transformations to any remaining out-stars, transforming them into loops.
\end{itemize}
In Figure~\ref{fig:5panels} the step labels correspond to the five steps indicated in the narrative. Although the reduction algorithm is stated generally for any initial $n$-grid -- regardless of initial resistance values-- we illustrate it with a simple numerical example.  Suppose, all edges have resistance equal to 1 in Step 1 (in the sequel an initial $n$--grid whose resistance values are identically one will be called an \textit{all one} $n$--grid). Then in Steps 2 and 3 all resistances are uniformly $Y(1,1,1)=\frac{1}{3}.$ 
In Step 4 the boundary edges have resistance value equal to $\frac{2}{3}$. In Step 5 the non-boundary edges have resistance $Y(\frac{1}{3},\frac{1}{3},\frac{1}{3}) = 1.$
\begin{figure}[ht!]
\begin{center}
\begin{tabular}{|c|c|c|c|c|}
\hline
\begin{tikzpicture}[yscale=.5,xscale=.5]

\node [below] at (3,3.2) {    }; 
 
\draw (0,0)--(1,2)--(2,0)--(0,0);
\draw (2,0)--(3,2)--(4,0)--(2,0);
\draw (4,0)--(5,2)--(6,0)--(4,0); 
\draw (1,2)--(2,4)--(3,2)--(1,2);
\draw (3,2)--(4,4)--(5,2)--(3,2);
\draw (2,4)--(3,6)--(4,4)--(2,4);

\node [below] at (3,6.2) {     };
\node [below] at (3,-.5) {Step 1};
\end{tikzpicture} 
& 
\begin{tikzpicture}[xscale=.5,yscale=.5] 
\draw  (1,1)--(1,2)--(2,3)--(2,4)--(3,5);
\draw  (5,1)--(5,2)--(4,3)--(4,4)--(3,5);
\draw (1,1)--(2,0)--(3,1)--(3,2)--(2,3);
\draw (3,1)--(4,0)--(5,1)--(5,2)--(4,3)--(3,2); 
 
\draw [dashed] (6,0)--(5,1);
\draw [dashed] (0,0)--(1,1);
 \draw [dashed] (3,5)--(3,6);
 
\node [below] at (3,-.5) {Step 2};
\end{tikzpicture}&
\begin{tikzpicture}[yscale=.5,xscale=.5]
 
\draw [dashed] (1,1)--(1,2)--(2,3)--
(2,4)--(3,5);
\draw [dashed] (5,1)--(5,2)--(4,3)--(4,4)--(3,5);
\draw [dashed] (1,1)--(2,0)--(3,1)--(4,0)--(5,1);
\draw (3,1)--(3,2)--(2,3)--(3,2)--(4,3);

\node [below] at (3,-.5) {Step 3};
 
\end{tikzpicture}&
\begin{tikzpicture}[yscale=.5,xscale=.5]
 
\draw (1,1)--(2,3)--(3,5);
\draw (5,1)--(4,3)--(3,5);
\draw (1,1)--(3,1)--(5,1);
\draw [dashed] (2,3)--(3,2)--(4,3);
\draw [dashed] (3,2)--(3,1); 
 
 \node [below] at (3,-.5) {Step 4};
\end{tikzpicture}&
\begin{tikzpicture}[yscale=.5,xscale=.5]
 
\draw (1,1)--(2,3)--(3,5);
\draw (5,1)--(4,3)--(3,5);
\draw (1,1)--(3,1)--(5,1);
\draw (2,3)--(4,3)--(3,1)--(2,3); 
 
 \node [below] at (3,-.5) {Step 5};
 
\end{tikzpicture}
\\ \hline
\end{tabular}

\end{center}
\caption{Illustration of the reduction algorithm, on a 3-grid.} \label{fig:5panels}
\end{figure}

Throughout this paper the label of edge $e, e \in \{L, R, B\}$ (standing, respectively, for the left, right, and base edges of a triangle in the upright oriented position), of the triangle in row $r$ diagonal $d$ is indicated by  $T_{r,d,e}.$ (Diagonals are in the left direction numbered from left to right as shown in Figure~\ref{fig:3grid}.)

 The notations $T^c_{r,d,e}$ and $T^c_{r,d}$ refer respectively to edges and triangles in an initial all-one $n$--grid  reduced $c$ times ($T^0$ refers to the initial $n$-grid).

Using these notations we can reformulate the reduction algorithm in terms of functions whose arguments are resistances in a parent grid and whose return values are resistances in a child grid. Tracing through Steps 1-5 of Figure \ref{fig:5panels} shows that starting with an all one   $3$-grid, $T,$ that
\begin{align*}
T^1_{1,1,L}&=
\Delta(T_{1,1,B},T_{1,1,L},T_{1,1,R}) +
\Delta(T_{2,1,L},T_{2,1,R},T_{2,1,B})\\
T^1_{2,1,L}&=
\Delta(T_{2,1,B},T_{2,1,L},T_{2,1,R}) +
\Delta(T_{3,1,L},T_{3,1,R},T_{3,1,B})
\end{align*}

This motivates defining the following function.

\begin{equation}\label{equ:leftboundaryedge}
T^{m+1}_{r,1,L}= \mathcal{P}(T^{m}_{r,1},T^m_{r+1,1}) =
 \Delta(T_{r,1,B}^{m},
            T_{r,1,L}^{m},
            T_{r,1,R}^{m})+
            \Delta(T_{r+1,1,L}^{m},
            T_{r+1,1,R}^{m},
            T_{r+1,1,B}^{m}).
\end{equation}
Here, $\mathcal{P}$ stands for the left-side perimeter of the grid. As pointed out
in  \cite{Sarajevo, EvansHendel}, because of the symmetries of the $m$-grid, knowledge
of edge values on the left half of the $m$--grid suffices to determine the edge values on its right half.

Besides the left-side perimeter function, only 3 other functions are needed 
to perform all computations \cite{Sarajevo, EvansHendel}. These functions compute
right-edges ($\mathcal{R}$), non-boundary left edges ($\mathcal{L}$), 
and base-edges ($\mathcal{B}$) in a once-reduced parent grid.

\begin{multline}\label{equ:baseedge}
T_{r,d,B}^{m+1} = \mathcal{B}(T_{r+2,d+1}^{m},T_{r+1,d}^{m},T_{r+1,d+1}^{m})=
    Y(\Delta(T_{r+2,d+1,L}^{m},
            T_{r+2,d+1,R}^{m},
            T_{r+2,d+1,B}^{m}),\\
            \Delta(T_{r+1,d,R}^{m},
            T_{r+1,d,B}^{m},
            T_{r+1,d,L}^{m}),
             \Delta(T_{r+1,d+1,B}^{m},
        T_{r+1,d+1,L}^{m},
            T_{r+1,d+1,R}^{m})),
\end{multline} 
\begin{multline}\label{equ:rightedge}
T_{r,d,R}^{m+1} =\mathcal{R}(T_{r,d}^{m},T_{r,d+1}^{m},T_{r+1,d}^{m})
	=    Y(\Delta(T_{r,d,R}^{m},
            T_{r,d,B}^{m},
            T_{r,d,L}^{m}),
      \Delta(T_{r,d+1,B}^{m},
        T_{r,d+1,L}^{m},
            T_{r,d+1,R}^{m},\\      
      \Delta(T_{r+1,d,L}^{m},
            T_{r+1,d,R}^{m},
            T_{r+1,d,B}^{m})),
\end{multline} 
    \begin{multline}\label{equ:leftedge}
T_{r,d,L}^{m+1} =\mathcal{L}(T_{r,d-1}^{m},T_{r,d}^{m},T_{r+1,d}^{m})=
    Y(\Delta(T_{r,d-1,R}^{m},
            T_{r,d-1,B}^{m},
            T_{r,d-1,L}^{m}),
     \Delta(T_{r,d,B}^{m},
        	T_{r,d,L}^{m},
            T_{r,d,R}^{m}),\\  
            \Delta(T_{r+1,d,L}^{m},
            T_{r+1,d,R}^{m},
            T_{r+1,d,B}^{m})).
        \end{multline}

Further details of these transformations can be found in~\cite{Sarajevo, EvansHendel}.
The following example from \cite{Sarajevo} motivates the circuit array.

\begin{example}\label{exa:row0} In Figure \ref{fig:5panels},   $T^1_{1,1,L}=\frac{2}{3}.$ This is true for any initial all one $n$--grid, $n \ge 2.$
A derivation similar to the one connected with Figure \ref{fig:5panels} shows that for an all-one $n$-grid, $n \ge 2,$ the boundary edges are uniformly $\frac{2}{3}$ while the remaining edges are uniformly 1.
Therefore, if $T$ is an all one $n$--grid, $n \ge 6,$  then 
$$T^2_{3,2,L} = \mathcal{L}(T^1_{3,1},T^1_{3,2},T^1_{4,2})=Y(
\Delta(1,1,\frac{2}{3}),\Delta(1,1,1),\Delta(1,1,1))=\frac{26}{27}.$$
Generally,  if $T$ is an all one $n$-grid, $n \ge 4s-2,$  then
\begin{equation}\label{equ:C0}T^s_{2s-1,s,L}=\frac{\frac{1}{3} 9^s-1}{\frac{1}{3}9^s}.\end{equation}
 The denominators of the sequence $T^s_{2s-1,s,L}, s \ge 1$ satisfy $G_{n+1}=9 G_n$ while
the numerators satisfy $G_{n+1}=9 G_n+8.$  These numbers form Row 0 in Table \ref{tab:circuitarray}. The proof of \eqref{equ:C0} is reviewed in Section \ref{sec:proofCX}.
\end{example}

\begin{definition}\label{def:circuitarray}
The circuit array 
$\{C_{i,j}\}$
is an infinite array such that for $0 \le i \le 2(j-1), j \ge 1,$ 
$$ 
C_{i,j}=T^{j}_{2j-1,j-\lfloor \frac{i+1}{2}\rfloor,LR}$$
where $LR$ refers to the left edge (right edge) if $i$ is even (odd) \cite{Sarajevo}. 
\end{definition}

Table \ref{tab:circuitarray} gives the numerical values of the first few rows and columns of the circuit array; they can be computed using the reduction algorithm or the four local functions introduced above. Further details are presented in \cite{Sarajevo}. Although this is not needed later, based on the Uniform Center Theorem \cite{EvansHendel}, it can be shown that Definition \ref{def:circuitarray} gives all numbers in column $j$ provided the initial all one $n$--grid satisfies  $n \ge 4j-2.$ Example \ref{exa:row0} illustrates this for Row 0.

\begin{center}
\begin{table}[!ht]
\begin{large}
\caption
{ First four rows and columns of the numerical values of the circuit array.}
\label{tab:circuitarray}
{
\renewcommand{\arraystretch}{2}
\begin{center}
\begin{tabular}{||c||r|r|r|r|||} 
\hline \hline

\;&$1$&$2$&$3$&$4$\\
\hline \hline

$0$&$\frac{2}{3}$&$\frac{26}{27}$&$\frac{242}{243}$&$\frac{2186}{2187}$\\

$1$&\;&$\frac{13}{12}$&$\frac{121}{120}$&$\frac{1093}{1092}$\\

$2$&\;&$\frac{1}{2}$&$\frac{89}{100}$&$\frac{16243}{16562}$\\

$3$&\;&\;&$\frac{1157}{960}$&$\frac{1965403}{1904448}$\\

\hline \hline
\end{tabular}
\end{center}
}
 \end{large}
 \end{table}
\end{center}
As pointed out in \cite{Sarajevo},  attempts to find recursive patterns in the rows of Table \ref{tab:circuitarray} failed; instead, it was suggested that it might be easier to find patterns in single or multiple variable polynomials and rational functions.

\begin{definition}\label{def:CX}
 To create the one-variable circuit array, $C^X,$ we rewrite the value of $C_{0,1} = T^{1}_{1,1,L}= \frac{2}{3}$ as a rational function, $\frac{X-3}{X}.$  We then proceed to calculate, using the four local functions, the resistance values of further reductions. \end{definition} 
 
 What results is a circuit array of rational functions in a single variable $X$ as shown in Table~\ref{tab:ircuitarrayx}. When the substitution $X=9$ is made we recover the numerical circuit array, Table \ref{tab:circuitarray}.

 
\begin{table}[!ht]
\caption{First four rows and columns of
$C^X,$
the one-variable circuit array.}
\label{tab:ircuitarrayx}
\renewcommand{\arraystretch}{2}
\begin{center}
\begin{tabular}{||c||r|r|r|r||} 
\hline \hline
\;&$1$&$2$&$3$&$4$\\

\hline \hline

$0$&$\frac{\frac{3}{9}X-1}{\frac{3}{9}X}$&$\frac{3X-1}{3X}$&$\frac{27X-1}{27X}$&$\frac{243X-1}{243X}$\\

$1$&\;&$\frac{3X - 1}{3X - 3}$&$\frac{27X-1}{27X-3}$&$\frac{243X-1}{243X-3}$\\

$2$&\;&$2\frac{1-\frac{4}{3}X+\frac{1}{3}X^2}{(X-1)^2}$&$2\frac{13-36X+39X^2}{(9X-1)^2}$&$2\frac{121-540X+3267X^2}{(81X-1)^2}$\\
 
$3$&\;&\;&$\frac{1}{12} \frac{3X-1}{9X-1} \frac{13-36X+39X^2}{1-2X+X^2}$&
$
\frac{1}{12} 
\frac{27X-1}{81X-1} 
\frac{121-540X+3267X^2}{10-36X+45X^2}$\\
\hline
\end{tabular}
\end{center}
\end{table}
 

\begin{definition}\label{def:CM}  The multivariable circuit array, $C^M,$ is constructed analogously to the construction of $C^X.$ We sequentially compute, one reduction at a time, the values in Table \ref{tab:circuitarray}  replacing certain numerical values with variables. More specifically, we set 
  $C_{0,1} = \frac{2}{3} = T^1_{2,1,L} = X_1$, 
$C_{2,2} = \frac{1}{2} = T^2_{3,1,L} =X_2,$  $C_{4,3}=\frac{13}{32} = T^3_{5,1,L} = X_3,$ 
$\ldots,C_{2i,2i-1} = T^i_{2i+1, 1,L}=X_{i}, i \ge 0.$
 The $\{X_i\}$ constitute  the \textit{left-side} diagonal of $C$ \cite{Sarajevo}. \end{definition}
 
 As a result of this construction, we obtain Table~\ref{tab:circuitarraymulti}. Entries in the first two rows of $C^M$ are functions of $X_1$, entries in the next two rows are functions of $X_1$ and $X_2$, etc.  As with $C^X$, we can recover the numerical circuit array by making the appropriate substitutions for $X_1, X_2,\ldots$.   

 {\small
\begin{table}[!ht]
\caption{First four rows and columns of $C^M,$  the multivariable circuit array}
\label{tab:circuitarraymulti}
\renewcommand{\arraystretch}{2}
\begin{center}
\begin{tabular}{||c||r|r|r|r||} 
\hline \hline
\;&$1$&$2$&$3$&$4$\\
\hline \hline
$0$&$X_1$&$\frac{X_1+8}{9}$&$\frac{X_1+80}{81}$&$\frac{X_1+728}{729}$\\

$1$&\;&$\frac{X_1+8}{3 X_1+6}$&$\frac{X_1+80}{3 (X_1+26)}$&$\frac{X_1+728}{3 (X_1+242)}$\\ 

$2$&\;&$X_2$&$\frac{9 X_1^2 X_2+8 X_1^2+36 X_1 X_2+128 X_1+36 X_2+512}{(X_1+26)^2}$&$\frac{81 X_1^2 X_2+80 X_1^2+324 X_1 X_2+2432 X_1+324 X_2+55808}{(X_1+242)^2}$\\ 

$3$&\;&\;&$\frac{(X_1+8) \left(9 X_1^2 X_2+8 X_1^2+36 X_1 X_2+128 X_1+36 X_2+512\right)}{3 (X_1+2) (X_1+26) (3 X_1 X_2+2 X_1+6 X_2+16)}$ & $\frac{(X_1+80) \left(81 X_1^2 X_2+80 X_1^2+324 X_1 X_2+2432 X_1+324 X_2+55808\right)}{3 (X_1+242) \left(27 X_1^2 X_2+26 X_1^2+108 X_1 X_2+596 X_1+108 X_2+5696\right)}$\\ 
\hline

\end{tabular}
\end{center}
\end{table}
}

  For $i \ge 0,$ we let $C_i = \{C_{i,j}\}_\text{\{all $j$ where defined\}}$ denote the $i$-th row of the circuit array.     $C^X_i$ and   $C^M_i$ are similarly defined.

\section{Annihilators}\label{sec:annihilators}

The main result of this paper identifies recursive patterns in the row sequences of  $C^X$ and $C^M.$  The motivation,  formulation, and proof of results,   require the use of the method of annihilators \cite{Boole, DeTempleWebb, annihilators}. Annihilator techniques were established by Boole in the first edition of~\cite{Boole}, are easy to use, and are a less messy approach to traditional inductive approaches in proofs connected with recursive sequences. Several authors \cite{DeTempleWebb, annihilators} have recently begun using them again when dealing with recursions.  This section reviews the basic theory.

\begin{definition} Given a sequence, $\mathbf{G}=\{G_s\}_{s \ge 0}$ (numerical or polynomial), the {translation operator}, $E,$ is defined for $s \ge 0, $ by $E(G_s)=G_{s+1}.$ \end{definition}

We let $1$ be the identity operator.  Scalar multiplication and addition of expressions in $E$ is defined componentwise. Multiplication of expressions in $E$ refers to composition. 

\begin{lemma} With $F$ any field (numerical or functional), $F[E]$ is an algebra over $F.$ (The elements of $F$ are called operators.) \end{lemma}

\begin{definition} An operator, $\mathbf{O},$ is an annihilator of a sequence $\mathbf{G}$ if 
$\mathbf{O(G)} \equiv \mathbf{0}$ (with $\mathbf{0},$ the 0 sequence). \end{definition}

 \begin{lemma} 
Any characteristic polynomial of a sequence, \textbf{G}, is also an annihilator of \textbf{G}.  \end{lemma}

\begin{example} Let $F=\{F_s\}_{s \ge 0},$ indicate the Fibonacci numbers, whose minimal  polynomial is $X^2 - X - 1.$ Then for $s \ge 0,$ $(E^2-E-1)(F_s) =
E^2(F_s)- E(F_s)-F_s =
F_{s+2}-F_{s+1}-F_s =0,$ the last equality following from the Fibonacci recursion. \end{example}

\begin{remark}   The ideal  
in $\mathbf{Z[X]}$  generated by the minimal polynomial $X^2-X-1$ consists of all characteristic polynomials of the Fibonacci numbers. For example, $X^3-2X^2+1$ is also a characteristic polynomial for the Fibonacci numbers. Similarly, annihilators need not be unique. \end{remark} 

To effectively use annihilators in proofs we will need two lemmas, one presenting trivial facts and another dealing with sums and products. The proofs of these lemmas are straightforward consequences of the fact that the annihilators form an algebra.

\begin{lemma}[The Trivial Lemma]   
(i)With $c$ an arbitrary constant, \ $E-c$ annihilates the sequence $\{c^s\}_{s \ge 0}.$
\\ (ii) For any constants $a,b, $ if $A$ annihilates the sequence $\textbf{G},$ then it also annihilates $a\mathbf{G}+b$ 
\end{lemma}

\begin{lemma}[The Addition-Multiplication Lemma]  
(i) Suppose $A_i, i \in I$ are annihilators of the sequences $\textbf{G}_i,i \in I.$ Then for
 arbitrary constants $c_i,$
$\sum_{i \in I} c_i G_i$  is annihilated by $\prod_{i \in I} A_i.$ \\
(ii)\cite{Jarden} Suppose $A_i,i=1,2$ annihilate the sequences $\mathbf{G}_i$ with characteristic polynomials, $p_i$ with roots  $r_{i,j}, 1 \le j \le k_i, i=1,2.$ Then $A= \prod_{\text{all pairs $(i,j),(i',j')$}} (E-r_{i,j} r_{i',j'}),$ annihilates the (term by term) product $\textbf{G}_1 \cdot \textbf{G}_2.$  
\end{lemma}

\begin{example}  Using these two lemmas, we may calculate the annihilator of the function sequence \textbf{G} = $\{G_s = (3^s+2^s)^2=9^s+4^s+2\cdot 6^s\}_{s \ge 0},$ as follows. (i)  $E-3$ and $E-2$ annihilate $\{3^s\}_{s \ge 0}$ and $\{2^s\}_{s \ge 0}$ respectively.   (ii) Hence their product, $(E-3)(E-2)$ annihilates the sum $3^s+2^s.$ (iii) Since $(E-3)(E-2)$ has roots $3,2$ we may obtain the annihilator of $(3^s+2^s)^2$ by taking into consideration all products of pairs of roots: $(E-9)(E-6)(E-4).$
\end{example}

\section{Main Results}\label{sec:proofCX}

This section   motivates and states the main result of this paper by presenting examples, detailing several patterns, stating the generalized pattern as the Main Conjecture, and indicating proven results (proof methods and special cases).  We first state and prove an inductive characterization of $C^X,$ followed by computational corollaries and illustrative examples.

To prepare for the next proposition we adopt the convention that the functions \eqref{equ:leftboundaryedge}-\eqref{equ:leftedge} can be stated with  either triangle or triangle-side arguments. For example, \eqref{equ:leftboundaryedge} can be functionally written as 
\begin{equation}\label{equ:trianglesidearguments}
	T^{m+1}_{r,1,L} = \mathcal{P}(T^{m}_{r,1}, T^{m}_{r+1,1})=
		 \mathcal{P}(T^{m}_{r,1,L},T^{m}_{r,1,R},T^{m}_{r,1,B},T^{m}_{r+1,1,L},T^{m}_{r+1,1,R},T^{m}_{r+1,1,B}),
\end{equation}
where we substitute for each triangle argument the set of its edges listed in clockwise order starting from the left edge and separated by commas.

Although Definition \ref{def:circuitarray} suffices  to compute entries of $C^X,$ the following proposition facilitates the computations.

\begin{proposition}\label{pro:circuitarrayinductive}  $C^X_{i,j}$ may be  computed using the inductive  relations (i) -- (iii), with the boundary conditions indicated by Parts (iv) and (v). 
\begin{enumerate}\item[(i)] For $c \ge 2,$ \begin{multline*}  C^X_{2(c-1),c} = 
\mathcal{P}\biggl(C^X_{2(c-2),(c-1)},
C^X_{2(c-2)-1,(c-1)},\\
C^X_{2(c-2)-1,(c-1)},
C^X_{2(c-2),(c-1)},
C^X_{2(c-2)-1,(c-1)},
C^X_{2(c-2)-1,(c-1)}
\biggr).
\end{multline*}\\ 
\item[(ii)] For $j > c \ge 2$,    
\begin{multline*} 
C^X_{2(c-1),j} = 
\mathcal{L}\biggl(
C^X_{2(c-1),j-1},
C^X_{2(c-1)-1,j-1} ,
C^X_{2(c-1)-1,j-1} ,
\; \\
C^X_{2(c-1)-2,j-1},
C^X_{2(c-1)-3,j-1},
C^X_{2(c-1)-3,j-1},
\;
C^X_{2(c-1)-2,j-1},
C^X_{2(c-1)-3,j-1},
C^X_{2(c-1)-3,j-1}  
\biggr). 
\end{multline*}
\item[(iii)] For $j \ge c \ge 2$, 
\begin{multline*}
C^X_{2(c-1)-1,j}=
 \mathcal{R}\biggl(
 C^X_{2(c-1)-2,j-1},
 C^X_{2(c-1)-3,j-1},
 C^X_{2(c-1)-3,j-1},
 \;     \\
C^X_{2(c-1-4),j-1},
C^X_{2(c-1)-5,j-1},
C^X_{2(c-1)-5,j-1},
\;
C^X_{2(c-1)-4,j-1},
C^X_{2(c-1)-5,j-1},
C^X_{2(c-1)-5,j-1}  
\biggr)
\end{multline*}
\item[(iv)] $C^X_{0,1}= \frac{X-3}{X}.$ 
\item[(v)] $C^X_i(j)=1, -3 \le i \le -1, \text{ all } j.$ 
\end{enumerate}
\end{proposition}

\begin{proof}
\textbf{Proof of (i).} 
 \begin{itemize}
\item By Definition \ref{def:circuitarray},  $C^X_{2(c-1),c} = T^j_{2c-1,1,L}.$
\item By \eqref{equ:leftboundaryedge}, $T^j_{2c-1,1,L} = \mathcal{P}(T^{j-1}_{2c-1,1}, T^{j-1}_{2c,1})$.
\item  We prove immediately below this bulleted list that for $c \ge 3$
\begin{equation}\label{equ:provenbelow} \mathcal{P}(T^{j-1}_{2c-1,1}, T^{j-1}_{2c,1})=\mathcal{P}(T^{j-1}_{2c-3,1}, T^{j-1}_{2c-3,1}).
\end{equation}
\item  By Definition \ref{def:circuitarray}, 
$C^X_{2(c-2),c-1} = T^{c-1}_{2c-3,1,L}$
and 
$C^X_{2(c-2)-1,c-1} = T^{c-1}_{2c-3,1,R}.$ We prove immediately below this bulleted list that
\begin{equation}\label{equ:provenbelow2} T^{c-1}_{2c-3,1,R}=T^{c-1}_{2c-3,1,B}.
\end{equation}
Combining this with \eqref{equ:provenbelow} and our notational conventions \eqref{equ:trianglesidearguments} completes the proof of Part (i).
\end{itemize}

We have left to prove 
\eqref{equ:provenbelow} and
\eqref{equ:provenbelow2}. Both these assertions follow from the Uniform Center Theorem \cite{EvansHendel} which we now briefly describe. The Uniform Center Theorem states that for sufficiently large $n,$   the triangles in the central region of the  diagonal $d=1,$ (in fact for rows $r \ge 3$) are uniformly labeled, that is corresponding sides have the same labels, \cite[Definition 6.2, Equation (12), and Theorem 6.5(a)]{EvansHendel}, and moreover, the right and base sides are identically labeled \cite[Theorem 6.5(d)]{EvansHendel}. This implies that corresponding sides of the  triangles $T^{c-1}_{r,1}, r \in \{2c-3,\dotsc, 2c\}$ are identically labeled and moreover their base and right-side labels are the same, proving  \eqref{equ:provenbelow} and
\eqref{equ:provenbelow2}.\\

\textbf{Proof of (ii)--(iii)} The proofs for Parts (ii) and (iii) are similar to the proof for Part (i) just presented, and hence they are
omitted.\\

\textbf{Proof of (iv).} This is simply Definition \ref{def:CX}\\

\textbf{Proof of (v).} The recursions for computing $C^X_r$ in Parts (i)--(iii) make use of the prior 3 or 4 rows depending on the parity of $r.$ Thus, these recursions suffice (assuming certain prior row or column values computed) for computation of $C^X_r,$ $r \ge 4.$ To prove these boundary conditions valid, we need to show that they can be used to compute rows 0,1,2,3.  We prove this for $r=0,$ the proofs for the other rows being similar and hence omitted. What has to be proven is that Part (ii) coupled with the boundary conditions of Part (v) correctly compute row 0. This proof for row 0 is similar to the proof of Part (i) .

\begin{itemize}
\item 
By Definition \ref{def:circuitarray}
$
C^X_{0,j} = T^j_{2j-1,j,L}, \qquad j \ge 2.
$
\item
By \eqref{equ:leftedge}, 
 $
T^j_{2j-1,j,L} = \mathcal{L}\bigl(
T_{2j-1,j-1}^{j-1},
T_{2j-1,j}^{j-1},
T_{2j,j}^{j-1}
\bigr).
$
\item We claim
 $
T^j_{2j-1,j,L} = \mathcal{L}\bigl(
T_{2j-1,j-1}^{j-1},
T_{2j-1,j}^{j-1},
T_{2j,j}^{j-1})
\bigr)=
\mathcal{L}\bigl(
T_{2j-3,j-1}^{j-1},
\mathbf{1},
\mathbf{1}
\bigr),
$
where $\mathbf{1}$ refers to a triangle whose sides are uniformly labeled 1.
This follows from the Uniform Center Theorem described above. First,
the following triangles are identically labeled: $T_{r,j-1}^{j-1}, r \in \{2j-3, 2j-2, 2j-1\}.$ Second, by 
\cite[Definition 2.13 with either Lemma 5.3 or Theorem 6.5(c)]{EvansHendel} the triangles 
$T_{2j-1,j}^{j-1}$ and $
T_{2j,j}^{j-1})$ are uniformly labeled 1.
\item But by Definition \ref{def:circuitarray},
$T_{2j-3,j-1}^{j-1} = C^X_{0,j-1},$ and by Part (v),  
$C^X_{r}=1, r \le -1.$ Combining the above shows that Part (iv), under the boundary conditions of Part (v), yields the correct computational formula, as was to be shown. 
\end{itemize}
\end{proof}

The power of this proposition lies in the ease by which it allows proofs of closed forms for the circuit array rows.
\begin{corollary}\label{cor:verification} For $c \ge 1:$\;
\begin{enumerate}\item[(i)] To prove $C^X_{2(c-1),j}=r_j(X), j \ge c,$ with $r_j(X)$ a rational function in $j$ and $X,$ it suffices to verify   that  $C^X_{2(c-1),c} = r_{c}(X),$ and $ C^X_{2(c-1),j} = r_j(X), j > c.$.
\item[(ii)] To prove $C^X_{2(c-1)-1,j}=r_j(X)$ it suffices to verify that $C^X_{2(c-1)-1,j} = r_j(X), j \ge c.$
\end{enumerate}
\end{corollary}

\begin{remark} In practice one uses the prior corollary in conjunction with Proposition \ref{pro:circuitarrayinductive} which computes row and column values $C^X$ using the four local functions, \eqref{equ:leftboundaryedge}-\eqref{equ:leftedge}, evaluated
at values that occur in prior rows or columns of $C^X.$. This requires proving that two rational functions are equal which is then conveniently accomplished by algebraic software.
\end{remark}

Using the corollary, we can immediately present closed forms for rows 0,1,2 of $C^X$ and row 2 of $C^M.$ (For $C^M$ we make explicit the verification since technically Proposition \ref{pro:circuitarrayinductive} applies to $C^X$ not to $C^M$. Although an analogous proposition could be presented for $C^M$ we only deal with rows 0--2 of $C^M$ and do not need elaborate machinery.)

\begin{corollary} For $C^X_0,$ \eqref{equ:C0}  (and the associated recursions) in Example \ref{exa:row0}   provides
a closed form. 
\end{corollary}

Similarly, we have
\begin{corollary}
 $C^X_{1,s}= c_1(s, X)=
\frac{3 \cdot 9^{s-2} X -1}{3 \cdot 9^{s-2} X}.$
\end{corollary}

\begin{corollary} Define $c_{2,0}(s)=27\cdot9^{s-3}-1,$
$c_{2,1}(s)= 24 \cdot (2s-3)  \cdot 9^{s-3}, $ and
$c_{2,2}(s)= 81 \cdot 81^{s-3} - 3 \cdot 9^{s-3}.$ Then $C^X_{2,s}= r_2(s,X)=
\frac{c_{2,0}(s) - c_{2,1}(s) X + c_{2,2}(X^2)}{(9^{s-2}X-1)^2}.$
\end{corollary}

\begin{remark} The functions $c(2,i), 0 \le i \le 2,$ are respectively the coefficients of $X^i$ in the numerator
of $r_2(s,X).$ The first
few terms are given by
$$ C_2^X = \left\{ 2 \frac{2(1-\frac{4}{3}X + \frac{1}{3}X^2)}{(X-1)^2},  2\frac{13-36X+39X^2}{(9X-1)^2}, 2\frac{121-540X+3267X^2}{(81X-1)^2}, \dotsc,\right\}.
$$ 
Besides this closed form, we may describe $C^X_2$ with recursions and annihilators.
The $c_{2,i}, 0 \le i \le 2,$ satisfy respectively the recursions, 
$G_n=10G_{n-1}-9G_{n-2},$ 
$G_n=18G_{n-1}+81G_{n-2},$ 
and $G_n=90G_{n-1}-81G_{n-2}.$ The corresponding annihilators are
$(X-1)(X-9),$ $(X-9)^2,$ and $(X-9)(X-81),$ respectively.  

Notice that the annihilators immediately suggest a pattern of products of linear factors of the form $X-9^k, k \ge 0.$   In contrast, neither the recursions nor the closed formula immediately suggest any pattern. For this reason the main
result of this paper is formulated in terms of annihilators. It turns out what was just seen in the second row appears to be true in 
all rows: the annihilators of the closed forms are products of linear factors of the form $X-9^k, k \ge 0.$ Before formulating this
result we show that (i) this pattern is true for $C^M,$ thus supporting the patterns noted and (ii) also motivate a stronger formulation of the patterns.
\end{remark}

\begin{corollary} For $s \ge 1,$ define
$ c^M_0(s,X_1) =\frac{  9^{s-1} -1+X_1}{  9^{s-1}},\; s  \ge 1.$ 
Then,  
$c^M_0(s,X_1) = c^M_{0,s}.$ 
\end{corollary}

\begin{remark} Technically, we cannot prove this corollary with Proposition \ref{pro:circuitarrayinductive} which is formulated
for $C^X$ not for $C^M.$  Rather than create an analogous proposition, we provide details for $C^M_2$ below and note here that the derivation for $C^M_i, i=0,1$ is similar and hence omitted.

Notice that the coefficients of the constant and $X_1$ term in numerators of $c^M_o(s,X_1)$ satisfy the recursions $G_s=G_{s-1}$ and 
$G_s=10G_{s-1} - 9 G_{s-2}$ respectively. No apparent pattern emerges. In contrast, the annihilators of the constant and $X_1$ term are $(E-1)$ and $(E-1)(E-9).$ This is consistent with our observation that the patterns for $C^X;$  are more naturally formulated in terms of annihilators. 
\end{remark}

\begin{corollary}
For $s \ge 2,$
$C^M_{1,s} = 
c^M_1(s,X_1)=
\frac{(9^{s-1}-1)+X_1}{3(X_1+3\cdot 9^{s-2}-1)}.$ 
\end{corollary}

\begin{remark}
The annihilators of the coefficient sequences of $1$ and $X_1$ are similar to those for  $C^M_0(s, X_1)$ and are therefore consistent with the idea of formulating patterns in terms of annihilators.
\end{remark} 

\begin{corollary}
Define coefficient functions, 
\begin{align*}
c_{2,1}(s)=9^{s-2} \\
c_{2,0}(s)=9^{s-2}-1\\
c_{1,1}(s)=4 \cdot 9^{s-2} \\
c_{1,0}(s) = 2+(16(s-2)-2) 9^{s-2} \\
c_{0,1}(s)=4 \cdot 9^{s-2} \\
c_{0,0}(s)=-1 + 9\cdot81^{s - 2} - (8 + 16 (s - 2)) 9^{s - 2}\\
c_{i,j}(s)=0, \text{otherwise}.
\end{align*}
By the Trivial and Addition-Multiplication Lemma of Section \ref{sec:annihilators} these coefficient functions are annihilated by the annihilators listed in  Table 
\ref{tab:XYLeft} and satisfy the recursions presented there.  

Define
$$
c^M_2(s,X_1, X_2)=
\frac{\displaystyle \sum_{0 \le i,j \le 2} c_{i,j}(s) X_1^i X_2^j}
{(3 \cdot 9^{s-2}-1+X_1)^2}.
$$
Then
\begin{equation}\label{equ:cms2} 
C^M_{2,s}= c_2^M(s,X_1,X_2).
\end{equation}
\end{corollary}
\begin{proof}
The proof is by induction.  The base case verifies 
\eqref{equ:cms2} for $s=3.$
 Using an induction assumption we complete the proof by verifying the following algebraic identity of rational functions: $$  c_2^M(s+1,X_1,X_2) = \mathcal{L}(c_2^M(s,X_1,X_2),c^M_1( X_1,s),c_0^M(s,X_1),1). $$
\end{proof}

\begin{center}
\begin{table}[ht!]
 \begin{small}
{
\renewcommand{\arraystretch}{1.3}
\begin{center}
\begin{tabular}{||c||c|c||} 
\hline \hline
  Coefficient & Recursion & Annihilator \\
 \hline
 $X_1^0X_2^0$ &
 $G_s=100G_{s-1}- 1638G_{s-2} + 8100G_{s-3}- 
 6561G_{s-4}$ &
 $(E-81) (E-9)^2 (E-1)$ \\
 
 $X_1^0X_2^1$ &
 $G_s=10 G_{s-1}-G_{s-2}$ &
 $(E-1)(E-9)$ \\
 
 $X_1^1X_2^0$ &
 $G_s=19G_{s-1}- 99G_{s-2} + 81G_{s-3}$ &
 $(E-9)^2 (E-1)$ \\
 
 $X_1^1X_2^1$ &
 $G_s=9 G_{s-1}$ &
 $(E-9)$ \\
 
 $X_1^2X_2^0$ &
 $G_s=10 G_{s-1}-9G_{s-2}$ &
 $(E-1)(E-9)$ \\
 
 $X_1^2X_2^1$ &
 $G_s=9 G_{s-1}$ &
 $(E-9)$ \\

\hline \hline
\end{tabular}
\end{center}
}
 \end{small} 

\caption{The recursions satisfied by the coefficients of the numerator of $C^{M}_{2,s}$.}\label{tab:XYLeft}
\end{table}
\end{center}

As remarked earlier, prior to stating the main results of this paper,  we present a stronger formulation of the patterns of the factorizations of annihlators of the rows $C^X_r.$  The following example provides motivation.

\begin{example}\label{exa:motivatestrong} Using either Proposition \ref{pro:circuitarrayinductive} or directly using Definition \ref{def:circuitarray}, we may calculate the following.
 
\begin{align*}
C^X_{3,4} &= & \frac{(-1 + 27 X) (121 - 540 X + 3267 X^2)}{24 (-1 + 81 X) (5 - 18 X + 
   45 X^2)} \\
C^X_{4,4}&=&\frac{(-1 + 27 X) (-89 + 333 X - 447 X^2 + 267 X^3)}{(5 - 18 X + 
   45 X^2)^2}\\
C^X_{5,4}&=&\frac{(13 - 36 X + 39 X^2) (-89 + 333 X - 447 X^2 + 267 X^3)}{192 (-1 + 
   X)^3 (5 - 18 X + 45 X^2)}
\end{align*}

Notice, that each $C^X_{i,4}, i \in \{3,4,5\}$  is a ratio of products of two polynomials (with an extra constant term). Moreover, the following factors are identical: 
the linear factor occurring in   $C^X_{3,4}, C^X_{4,4},$  the cubic factors occurring in $C^X_{4,4}, C^X_{5,4},$ and the quadratic factor occurring in the denominators of
$C^X_{i,4}, i \in \{3,4,5\}$ with this factor being squared in $C^X_{4,4}$.

It turns out that these factors when restricted to any particular row have coefficients that satisfy recursions whose underlying annihilator consists of linear factors of the form $X-9^k.$ Remarkably, we can completely describe these factorizations using a 3-dimensional array satisfying order 1 (albeit non-homogeneous) recursions. This suggests a stronger form of the main result. The main results of this paper are stated as conjectures with a strong and weak form.  
\end{example}

\begin{conjecture}[Main Result, Weak Form] 
 For each row $r \in \{2(c-1), 2(c-1)-1\},$ each of the sequences $Num[C^X_r(j)], Den[C^X_r(j)],
Num[C^M_r(j)], Den[C^M_r(j)], j \ge c $   satisfies a recursion whose corresponding annihilators are products of linear factors (possibly with repetition) of the form $X-9^k, k \ge 0$.\\
\end{conjecture}

\begin{conjecture}[Main Result, Strong Form] For $c=1,r=0$ and for each $c \ge 2$ and corresponding $r \in \{2(c-1), 2(c-1)-1\}$ there exists a sequence of polynomials  $p_r(j,X), j \ge c$ such that 
\begin{enumerate}
    \item[(i)] If $c \ge 1,$ and $r=2(c-1)-1$ then there exists a constant, $K,$ depending at most on $c,$ with 
$$
C^X_{r,j} = K \frac{p_{r-3}(j-1) p_{r-1}(j)}{p_{r-2}(j) p_r(j)}.
$$
\item[(ii)] If $c \ge 1,$ and $r=2(c-1)$ then there exists a constant, $K,$ with
$$
C^X_{r,j} = K \frac{p_{r-4}(j-1) p_{r}(j)}{p_{r-1}(j)^2}.
$$
Moreover, we can completely describe the factorizations of the annihilators of the coefficient sequences of the $p_r.$ Towards this end we first define the following  3-dimensional recursive array, $e$  (standing for exponent) the first four rows of which are presented  in Table~\ref{tab:arraye}.   
\begin{enumerate}
\item[(a)]$e_{i,j,k}=0, \text{ otherwise (i.e., if not defined by the remaining equations)},$
\item[(b)]$e_{i,0,i}=1, \qquad i \ge 0,$ 
\item[(c)]$e_{i+1,0,k} = e_{i,0,k}+k, \qquad  i \ge 0, 0 \le k \le i,$ 
\item[(d)]\;$e_{i,i+1,k+1} = e_{i,0,k},  \qquad i \ge 0,  0 \le k \le i,$
\item[(e)]$e_{i,i,k} = e_{i,i-1,k}+k, \qquad i \ge 1, 1 \le k \le i, $
\item[(f)]$e_{i,j,k}=e_{i-1,j,k}+k, \qquad i >j, 1  \le j \le i, 1 \le k \le i.$\end{enumerate}

\item[(iii)] For $c \ge 1,$ and  $r \in \{2(c-1), 2(c-1)+1\},$ $deg(p_r) = c.$ For each $j,  0 \le j \le c,$ the degree-$j$ sequence, that is, the sequence of coefficients of $X^j$ in $p_r(X)$, satisfies a recursion whose annihilators are as follows: 
\begin{itemize}
    \item
The annihilator of the degree-0 coefficient sequence is $\prod_{k=0}^c (X-9^k)^{e_{2r,0,k}}.$
\item
The annihilator of the degree-$j$ coefficient sequence, $j >0,$ is $\prod_{k=1}^c (X-9^k)^{e_{2r,j,k}}.$ 
\end{itemize}
\end{enumerate}
\end{conjecture}
\begin{center}
\begin{table}[!ht]
\caption
{ First 4 rows and columns of the  3-dimensional array $e.$
Rows and columns start at index 0, while within a given cell, indices start at 0 (respectively 1) for column 0 (respectively column $j, j >0$).}

\label{tab:arraye}
{
\renewcommand{\arraystretch}{1.5}
\begin{center}
\begin{tabular}{||c||l|l|l|l|||} 
\hline \hline
Row index, $r$ & 
Coefficient of $X^0$ & 
Coefficient of $X^1$ &
Coefficient of $X^2$ &
Coefficient of $X^3$ \\
\hline
\; &
Starts at $(X-1)^e$ &
Starts at $(X-9)^e$ &
Starts at $(X-9)^e$ &
Starts at $(X-9)^e$ \\
\hline \hline
0 &
1 &
1 &
\; &
\; \\
1 &
1,1 &
2 &
1,1 &
\; \\
2 &
1,2,1 &
3,2 &
2,3&
1,2,1 \\
3 &
1,3,3,1 &
4,4,2 &
3,5,3 &
2,4,4 \\
\hline \hline
\end{tabular}
\end{center}
}
 \end{table}
\end{center}

\begin{example}
To illustrate Part (iii)
 of the strong form of the main conjecture we calculate, for $c \ge 3,$ the annihilator of $X$ of the cubic polynomial,  $p_r, r \in \{4,5\}.$
 By Table \ref{tab:arraye},
 ${e_{3,1,i}}_{\{i=1,2,3\}} = \{4,4,2\}.$
By Part (iii) of the Strong Form of the Main Conjecture the coefficient sequence of
of $X$ in $\{p_3(i)\}_{i\ge 3}$ is annihilated by $(X-9)^4 (X-81)^4 (X-729)^2.$
 
  The illustrations for Parts (i) and (ii) are similar so we suffice  with an illustration of Part (ii) using the explicit value of $C^X_{5,4}$ presented in Example \ref{exa:motivatestrong}. Using the notation of Part (ii), for $C^X_{5,4}$
we have $r=5, c=3, j=4.$ To show that Part (ii) describes the right hand side of the equation with $C^X_{5,4}$ we must compute 
$p_{r-3}(j-1), p_{r-2}(j), p_{r-1}(j),$ and show that the cubic $p_r(j)$ is a factor of $C^X_{5,j}.$  

Using Proposition \ref{pro:circuitarrayinductive} we may compute
\begin{itemize}
\item $\{C^X_{2,k}\}_{\{k \ge 2\}} = \left\{\frac{2(X-3)}{3(X-1)},
2\frac{13-36X+39X^2} {(1-9X)^2},2\frac{121-540X+3267X^2} {(1-81X)^2},\dotsc\right\} $ confirming the value of $p_{r-3}(j-1) = 13-36X+39X^2.$
\item $\{C^X_{3,k}\}_{\{k \ge 3\}} = \left\{\frac{(3 X-1) \left(39 X^2-36 X+13\right)}{12 (X-1)^2 (9 X-1)},\frac{(27 X-1) \left(3267 X^2-540 X+121\right)}{24 (81 X-1) \left(45 X^2-18 X+5\right)},\frac{(243 X-1) \left(265599 X^2-6804 X+1093\right)}{12 (729 X-1) \left(7371 X^2-486 X+91\right)}, \ldots\right\}$, confirming the value of
$p_{r-2}(j) = 5-18X+45X^2.$
\item $\{C^X_{4,k}\}_{\{k \ge 3\}}=\left\{\frac{(X-3) (3 X-1)}{6 (X-1)^2},\frac{(27 X-1) \left(267 X^3-447 X^2+333 X-89\right)}{4 \left(45 X^2-18 X+5\right)^2},\frac{(243 X-1) \left(438561 X^3-187029 X^2+87399 X-16243\right)}{2 \left(7371 X^2-486 X+91\right)^2}, \ldots\right\}$ confirming the value of 
$p_{r-1}(j) = -89+333X-447X^2+267X^3.$
\item $\{C^X_{5,k}\}_{\{k \ge 4 \}}=
\left\{\frac{\left(39 X^2-36 X+13\right) \left(267 X^3-447 X^2+333 X-89\right)}{192 (X-1)^3 \left(45 X^2-18 X+5\right)}, \frac{\left(3267 X^2-540 X+121\right) \left(438561 X^3-187029 X^2+87399 X-16243\right)}{192 \left(7371 X^2-486 X+91\right) \left(981 X^3-855 X^2+495 X-109\right)}, \ldots\right\},
$ confirming that $(X-1)^3$ is a cubic factor (of the denominator) of $C^X_{5,4}$ 
\end{itemize}

Using the above lists we can also verify, consistent with Part (ii) of the Strong Form of the
Main Conjecture, that $C^X_{5,5}= K\frac{p_2(5) p_4(5)}{p_3(5) p_5(5)} = \frac{\left(3267 X^2-540 X+121\right) \left(438561 X^3-187029 X^2+87399 X-16243\right)}{192 \left(7371 X^2-486 X+91\right) \left(981 X^3-855 X^2+495 X-109\right)},$ with $K= \frac{1}{192}.$

\end{example}

It is natural to ask what can and cannot currently be proven as well as what evidence we have for believing   the Main Conjectures. Using the fact that the product of the  characteristic polynomials of two recursive sequences has as its root products of roots of the individual characteristic  polynomials, \cite{Jarden}, we have the following elementary result.

\begin{corollary} The Strong Form of the Main Conjecture implies the Weak Form of the Main Conjecture. \end{corollary}

Besides Corollaries 4.4-4.6 we have the following.
\begin{corollary} The Strong Form of the Main Conjecture is true for rows $r, 0 \le r \le 7.$
\end{corollary}
\begin{proof}
We can obtain explicit forms for the $p_r, 0 \le r \le 7.$ We then simply apply Corollary \ref{cor:verification}. 
\end{proof}

This proof sheds light on what we can't do. Although we conjecture that each entry of $C^X_{r,j}$ is a quotient of a numerator and denominator of products of two polynomials, and although we can explicitly give the underlying recursions satisfied by the coefficients of these polynomials, we have not found explicit forms for all entries; equivalently, we have not found explicit patterns for the initial conditions. It is this which is holding up a complete proof. However we believe the conjecture true because of the many initial cases we can prove as well as the very transparent and simple  patterns underlying them.

We also note that because of the computational complexity involved we have not pursued $C^M$ further beyond Corollaries 4.8 and 4.10. We do however believe that at least the weak form of the  conjecture (with $C^M$ replacing $C^X$ in the weak form of the Main Conjecture) is true. 

\section{Conclusion}

This paper has explored the polynomial sequences of $C^X$ which is a  polynomial generalizations of the numerical circuit array, which presents  resistance distances  occurring in the central regions of the collection of $n$-grids and their reductions. Interesting simply stated patterns emerge some of which we have proved. We believe this to be a fertile ground for future research.


\end{document}